\documentclass[10pt]{amsart}
\usepackage{amsmath,amssymb,amsfonts,a4wide}
\usepackage{epsfig}

\newtheorem{theorem}{Theorem}

\def\be{\begin{equation}}
\def\ee{\end{equation}}

\newtheorem{lemma}[theorem]{Lemma}
\newtheorem{proposition}[theorem]{Proposition}

\begin{document}

\author[H\"{u}seyin Acan]{H\"{u}seyin Acan}\address{School
  of Mathematical Sciences, Monash University, Melbourne, VIC 3800,
   Australia} 
\email{huseyin.acan@monash.edu.au} 

\author[Pawe{\l} Hitczenko]{Pawe{\l} Hitczenko${}^\ddagger$}
\thanks{$\ddagger$ The second author was partially supported by a 
Simons Foundation grant \#208766. His work was carried out during a visit at Monash
University in the first half of 2014. 
He would like to thank the members of the School of Mathematical
Sciences and Nick Wormald in particular for  hospitality  and 
support.}
\address{Department of Mathematics, Drexel University, Philadelphia, 
PA  19104, USA} 
\email{phitczenko@math.drexel.edu}

\title[Covariances of outdegrees in  plane recursive trees]{On the covariances of outdegrees in random plane recursive trees}


\maketitle

%
%
%
%
%
%
%
%
%
%

In 2005 Janson \cite{SJ_s}, extending earlier work of Mahmoud, Smythe,
and Szyma\'nski \cite{mss}, established the joint asymptotic normality of the 
outdegrees of a random plane recursive tree  (we refer to \cite{SJ_s} for  references, 
discussion and statements, and to \cite{SJ_l} for a much wider
context). 
 In particular, he gave the following formula for the entries of the
limiting covariance matrix  \cite[Theorem~1.3]{SJ_s}:  
\begin{equation}\label{sj_sig}\tilde\sigma_{ij}=2\sum_{k=0}^i\sum_{l=0}^j\frac{(-1)^{k+l}}{k+l+4}{i\choose
k}{j\choose l}\left(\frac{2(k+l+4)!}{(k+3)!(l+3)!}-1-\frac{(k+1)(l+1)}{(k+3)(l+3)}\right).\end{equation}
Since this formula is not very convenient to work with (in particular the behavior
of $\tilde\sigma_{ij}$ as $i$ and/or $j$ grow to infinity is not
immediately clear),  we found it worthwhile to point out that it may be
considerably simplified. Throughout, $(x)_m=x(x-1)\dots(x-(m-1))$ denotes the falling factorial.
\begin{proposition}\label{prop1} 
For all integers $i\ge0$, $j\ge 0$ we have
\begin{align*}\tilde\sigma_{ij}&= \frac{16}{(i+3)_3(j+3)_3}-\frac{24}{(i+j+4)_4},\mbox{\ if $i\ne
     j$;}\\
\tilde\sigma_{jj}&=\frac4{(j+3)_3}+\frac{16}{(j+3)_3^2}-\frac{24}{(2j+4)_4}.\end{align*}
\end{proposition}
For the proof 
we will need two identities involving binomial coefficients that we
present in the following two lemmas.
\begin{lemma}\label{lem1}
For all integers $k\ge 0$, $a\ge0$, and $j\ge k$:
\[
\sum_{l=0}^j(-1)^l{j\choose l}{k+l+a\choose
l+a}=\left\{\begin{array}{rr}0,&\mbox{ if $j>k$};\\ (-1)^j;&\mbox{\
   if $j=k$}.\end{array}\right.\]
\end{lemma}
\begin{proof} This is a special case of formula (5.24) in
  \cite{GKP} as we have found thanks to the  encouragement by one of the
  referees to search for a source in the literature. It corresponds to
  $m=0$ and $s=n+a$ in the 
  notation used in \cite{GKP}. However, to keep this letter self--contained we supply a short
  proof. We proceed  by induction over $k$ for all $a$ and $j\ge k$. If
$k=0$ the equality holds for all $a\ge0$ since its left--hand side is $(1-1)^j$ if $j>0$ and $1$ if
$j=0$. Assume it holds for non-negative integers up to
$k$ and all values of $a$ and $j\ge k$. Let $a\ge0$ be any
integer. For $j\ge k+1$
\begin{align*}&\sum_{l=0}^j(-1)^l{j\choose l}{k+1+l+a\choose
l+a}=\frac{k+1+a}{k+1}\sum_{l=0}^j(-1)^l{j\choose l}{k+l+a\choose
l+a}\\&\quad+\sum_{l=0}^j(-1)^l\frac{j!}{l!(j-l)!}\frac{l(k+l+a)!}{(k+1)!(l+a)!}.
\end{align*}
The first sum is zero by the inductive hypothesis. We cancel the $l$'s
in the second sum and write it as
\[
\sum_{l=1}^{j} (-1)^l\frac{j!}{(l-1)!(j-l)!}\frac{(k+l-1+a+1)!}{(k+1)k!(l-1+a+1)!}
=-\frac j{k+1}\sum_{l=0}^{j-1}(-1)^l {j-1\choose l}{k+l+a+1\choose l+a+1}.
\]
By the inductive hypothesis (applied to  $k$, $a+1$, and $j-1$) this sum
is zero if $j-1>k$ and is $(-1)^{j-1}$ if $j-1=k$. This proves that
the original expression is zero if $j>k+1$ and is $(-1)^j$ if $j=k+1$
thus completing the induction.
\end{proof}
\begin{lemma}\label{lem2} For all  integers $j\ge0$, $i\ge0$, and $a\ge1$ we have 
\[\sum_{l=0}^j\frac{(-1)^l}{(l+a){l+a+i\choose
     i}}{j\choose l}=\frac1{a{i+j+a\choose
     a}}=\frac{(a-1)!}{(i+j+a)_a}.\]
\end{lemma}
\begin{proof} We use induction over $j\ge0$ for all $a\ge1$
 and $i\ge0$. (Alternatively $i$ can stay fixed throughout). When
 $j=0$ both sides  are $1/(a{a+i\choose i})$. Assume
 the statement 
holds for all integers up to $j$ and all $a\ge1$. We
 will prove that it holds for $j+1$ and all integers $a\ge1$. We have
\begin{align*}&\sum_{l=0}^{j+1}\frac{(-1)^l}{(l+a){l+a+i\choose
     i}}{j+1\choose l}=   
\sum_{l=0}^{j+1}\frac{(-1)^l}{(l+a){l+a+i\choose
    i}}\left\{{j\choose l}+{j\choose
    l-1}\right\}\\&\quad=
\sum_{l=0}^{j}\frac{(-1)^l}{(l+a){l+a+i\choose
    i}}{j\choose l}+
\sum_{l=1}^{j+1}\frac{(-1)^l}{(l+a){l+a+i\choose
   i}}{j\choose
   l-1}\\&\quad=
\frac1{a{i+j+a\choose a}}+
\sum_{l=1}^{j+1}\frac{(-1)^{l-1+1}}{(l-1+a+1){l-1+a+1+i\choose
  i}}{j\choose
  l-1}\\&\quad=
\frac1{a{i+j+a\choose a}}-
\sum_{l=0}^{j}\frac{(-1)^{l}}{(l+a+1){l+a+1+i\choose
  i}}{j\choose
  l}=\frac1{a{i+j+a\choose a}}-\frac1{(a+1){i+j+a+1\choose a+1}}
\\&\quad=\frac{(a-1)!(i+j)!}{(i+j+a)!}\left\{1-\frac
 a{i+j+a+1}\right\}=
\frac{(a-1)!(i+j+1)!}{(i+j+a+1)!}=\frac1{a{i+j+1+a\choose a}},
\end{align*}
where we have used the inductive hypothesis, first with $j$ and $a$
and then with $j$ and $a+1$. This proves Lemma~\ref{lem2}.
\end{proof}
\noindent{\bf Proof of Proposition~\ref{prop1}.} 
Assume without loss of generality that $0\le i\le j$.  We split the
right--hand side
of \eqref{sj_sig} as 
\begin{align}\label{1stpart}&4\sum_{k=0}^i\sum_{l=0}^j\frac{(-1)^{k+l}}{k+l+4}{i\choose
k}{j\choose
l}\frac{(k+l+4)!}{(k+3)!(l+3)!}
\\&\quad\label{2ndpart}
-2\sum_{k=0}^i\sum_{l=0}^j\frac{(-1)^{k+l}}{k+l+4}{i\choose
k}{j\choose
l}\left(1+\frac{(k+1)(l+1)}{(k+3)(l+3)}\right).\end{align}
We claim that \eqref{1stpart} is zero unless $i=j$ in which case
it is $4/(j+3)_3$. 
To see this note that
\[\frac{(k+l+4)!}{(k+l+4)(k+3)!(l+3)!}=\frac1{(k+3)_3}{k+l+3\choose
 l+3},\]
so that
\begin{align*}&
\sum_{k=0}^i\sum_{l=0}^j\frac{(-1)^{k+l}}{k+l+4}{i\choose
 k}{j\choose l}\frac{(k+l+4)!}{(k+3)!(l+3)!}
\\&\quad=\sum_{k=0}^i\frac{(-1)^k}{(k+3)_3}{i\choose
 k}\sum_{l=0}^j(-1)^l{j\choose l}{k+l+3\choose l+3}.\label{non-zero}\end{align*}
Since $k\le i$ and we  assumed  that $i\le j$, by Lemma~\ref{lem1}, the inner
sum is zero unless
$i=j$ and if that is the case only the term $k=i=j$ in the outer sum is non--zero
and it is
\[\frac{(-1)^j}{(j+3)_3}{j\choose j}\sum_{l=0}^j(-1)^l{j\choose
 l}{j+l+3\choose l+3}=\frac{(-1)^{2j}}{(j+3)_3}=\frac1{(j+3)_3},
\]
by Lemma~\ref{lem1}.
To handle \eqref{2ndpart} we write
\[1+\frac{(k+1)(l+1)}{(k+3)(l+3)}=2\frac{(k+1)(l+1)+(k+l+4)}{(k+3)(l+3)},\]
so that \eqref{2ndpart} is
\begin{align}
&-4
\sum_{k=0}^i(-1)^k\frac{k+1}{k+3}{i\choose
k}\sum_{l=0}^j(-1)^l\frac{l+1}{(l+3)(k+l+4)}{j\choose
l}\label{2ndline}
\\&\quad
\label{1stline}
-4\sum_{k=0}^i(-1)^k\frac1{k+3}{i\choose
 k}\sum_{l=0}^j(-1)^{l}\frac1{l+3}{j\choose l}.
\end{align}
By Lemma~\ref{lem2} (used with $a=3$ and $i=0$) 
\eqref{1stline} is
\[-4\frac2{(i+3)_3}\frac2{(j+3)_3}=-\frac{16}{(i+3)_3(j+3)_3}.
\]
To handle \eqref{2ndline} we first note that  
\[\sum_{l=0}^j(-1)^l\frac{l+1}{(l+3)(k+l+4)}{j\choose
l}=\frac{k+3}{(k+1)(k+4){k+j+4\choose j}}-\frac2{3(k+1){j+3\choose
j}}.\]
This follows from partial fraction decomposition 
\[\frac{l+1}{(l+3)(k+l+4)}=\frac{k+3}{k+1}\cdot\frac1{k+l+4}-\frac2{(k+1)(l+3)}
\] 
and 
\[
\sum_{l=0}^j\frac{(-1)^l}{k+l+4}{j\choose l} = \frac1{(k+4){k+j+4\choose j}},\quad
\sum_{l=0}^j\frac{(-1)^l}{l+3}{j\choose l} = \frac1{3{j+3\choose j}}, 
\]
which is Lemma~\ref{lem2} used twice, with $a=k+4$ and $i=0$ for the first equality, and with $a=3$ and $i=0$ for the second equality.  
Therefore,  \eqref{2ndline} is
\[-4\sum_{k=0}^i(-1)^k\frac1{(k+4){k+j+4\choose j}}{i\choose k}
+\frac{16}{(j+3)_3}\sum_{k=0}^i(-1)^k\frac1{k+3}{i\choose k}.\]
Applying Lemma~\ref{lem2} (with $a=4$ and general $i$) to the first term
and with $a=3$ and $i=0$ to the second term we find that
\eqref{2ndline} is 
\[-\frac{24}{(i+j+4)_4}+\frac{32}{(i+3)_3(j+3)_3}.\]
Hence, the combined  contribution of \eqref{2ndline} 
and \eqref{1stline} is
\[-\frac{16}{(i+3)_3(j+3)_3}+\frac{32}{(i+3)_3(j+3)_3}-\frac{24}{(i+j+4)_4}=\frac{16}{(i+3)_3(j+3)_3}-\frac{24}{(i+j+4)_4},\]
which completes the proof.

\bibliographystyle{plain}

\end{document}